\documentclass[12pt,a4paper,draft]{article}
\usepackage{amsmath,amsfonts,amssymb,amscd,amsthm}
\usepackage{latexsym}
\usepackage[T2A]{fontenc}
\usepackage[cp1251]{inputenc}
\usepackage{graphicx}

\theoremstyle{plain} 
\newtheorem{theorem}{Theorem}[section]
\newtheorem{lemma}[theorem]{Lemma}
\newtheorem{corollary}[theorem]{Corollary}

\newtheorem{proposition}[theorem]{Proposition}

\theoremstyle{definition} 
\newtheorem{definition}[theorem]{Definition}
\theoremstyle{remark} 
\newtheorem{remark}[theorem]{Remark}

\newcommand{\mcJ}{\mathcal{J}}
\newcommand{\bbC}{\mathbb{C}}
\newcommand{\bbCP}{\mathbb{CP}}
\newcommand{\bbN}{\mathbb{N}}

\newcommand{\bfa}{\mathbf{a}}

\newcommand{\bfz}{\mathbf{z}}
\newcommand{\bfm}{\mathbf{m}}

\newcommand{\Rat}{\mathrm{Rat}}
\newcommand{\PSL}{\mathrm{PSL}}

\begin{document}
\title{Algebraic independence of multipliers of periodic orbits in the space of rational maps of the Riemann sphere}
\author{Igors Gorbovickis}
\maketitle

\begin{abstract}
We consider the space of degree $n\ge 2$ rational maps of the Riemann sphere with $k$ distinct marked periodic orbits of given periods. First, we show that this space is irreducible. For $k=2n-2$ and with some mild restrictions on the periods of the marked periodic orbits, we show that the multipliers of these periodic orbits, considered as algebraic functions on the above mentioned space, are algebraically independent over~$\bbC$. Equivalently, this means that at its generic point, the moduli space of degree $n$ rational maps can be locally parameterized by the multipliers of any $2n-2$ distinct periodic orbits, satisfying the above mentioned conditions on their periods.  
This work extends previous similar result obtained by the author for the case of complex polynomial maps.
\end{abstract}

\section{Introduction}
Let $\Rat_n$ denote the space of degree $n$ rational maps of the Riemann sphere. The moduli spaces of degree $n$ rational maps $M_n$ is the space $\Rat_n$ modulo the action by conjugation of the group of M\"obius transformations,
$$
M_n=\Rat_n/\PSL_2(\bbC).
$$  
A key point in studying the moduli spaces $M_n$ is the choice of a parameterization. The idea of using the multipliers of the fixed points of a map as the parameters of the moduli space appears naturally in many works on the subject. Notably, in~\cite{Milnor_M2} J.~Milnor used the multipliers of the fixed points to parameterize the moduli space of degree $2$ rational maps. Using this parameterization he proved that this moduli space is isomorphic to $\bbC^2$.

It is not hard to see that the same approach does not work for $n\ge 3$, since a degree $n$ rational map does not have enough fixed points. Indeed, $\dim M_n=2n-2$, while a map $f\in\Rat_n$ has only $n+1$ fixed points (counted with multiplicity). In addition to that, the multipliers of these $n+1$ fixed points satisfy a certain relation, namely, the holomorphic index formula (see~\cite[Section~12]{Milnor}), hence they cannot be independent parameters.

In order to overcome this difficulty, instead of the multipliers at the fixed points one can try to use the multipliers of periodic orbits as the local parameters on the moduli space $M_n$. It is not hard to see that the map from $M_n$ to the multipliers of the chosen periodic orbits is defined in a neighborhood of a generic point of $M_n$. The main difficulty is to check whether this map is a local diffeomorphism, when the number of the chosen periodic orbits is equal to the dimension of $M_n$. Since multipliers are (multiple valued) algebraic maps on $\Rat_n$, this leads to the question whether there exist ``hidden'' algebraic relations between the multipliers of the chosen periodic orbits. 
In other words, are the chosen multipliers algebraically independent over~$\bbC$, if we view those multipliers as (multiple valued) functions on $\Rat_n$?

In~\cite{McMullen} McMullen proved that if $n\ge 2$ then, except for the flexible Latt\`{e}s maps, an element of $M_n$ is determined up to finitely many choices by the multipliers of \textit{all} of its periodic orbits. This implies that one can always choose $2n-2=\dim(M_n)$ distinct periodic orbits whose multipliers, considered as (multiple valued) functions on $\Rat_n$, are algebraically independent over~$\bbC$.

In this paper we prove the following theorems:
\begin{theorem}\label{main1}
For $n=2$, the multipliers of any two distinct periodic orbits considered as (multiple valued) algebraic functions on $\Rat_2$, are algebraically independent over $\bbC$. In other words, they do not satisfy any polynomial relation with complex coefficients.
\end{theorem}

\begin{theorem}\label{main2}
For $n\ge 3$, the multipliers of any $2n-2$ distinct periodic orbits, considered as (multiple valued) algebraic functions on $\Rat_n$, are algebraically independent over~$\bbC$, if  the following two conditions are simultaneously satisfied:

(i) no more than $n$ of these orbits have period $1$,

(ii) at least one of these orbits has period greater than $2$ and periods of all other orbits are not equal to $2$.
\end{theorem}

\begin{remark}
If the first condition in Theorem~\ref{main2} is not satisfied then this means that the chosen collection of periodic orbits contains $n+1$ fixed points, whose multipliers are related by the holomorphic index formula and hence, cannot be independent. Thus, the first condition in Theorem~\ref{main2} cannot be removed. On the other hand, we conjecture that the second condition in Theorem~\ref{main2} can be significantly weakened or even completely removed.
\end{remark}

As an important corollary for the theory of rational maps, we deduce existence of hyperbolic components, characterized by $2n-2$ attracting periodic orbits with periods satisfying Theorem~\ref{main1} or Theorem~\ref{main2}.

\begin{corollary}
For every tuple of $2n-2$ periods, such that either $n=2$, or $n\ge 3$ and the two conditions of Theorem~\ref{main2} are satisfied, there exists a hyperbolic component in the moduli space $M_n$, characterized by $2n-2$ attracting periodic orbits of the given periods. 
\end{corollary}
\begin{proof}
If the considered $2n-2$ multipliers are algebraically independent, then the algebraic map from $\Rat_n$ to these multipliers maps $\Rat_n$ to a Zariski open subset of $\bbC^{2n-2}$. In particular, this means that there exists a map $f\in\Rat_n$ with $2n-2$ attracting periodic orbits of considered periods. This implies the existence of the corresponding hyperbolic component.
\end{proof}

Finally, we mention that for the case of polynomial maps the theorem analogous to Theorems~\ref{main1} and~\ref{main2} is proved by the author in~\cite{mult_n}. This theorem states that for $n\ge 3$, the multipliers of any $n-1$ distinct periodic orbits considered as algebraic functions on the space of all degree~$n$ polynomials, are algebraically independent over~$\bbC$. Unlike Theorem~\ref{main2}, the above formulated theorem does not have any additional restrictions on the periods of the orbits, and since the moduli space of degree~$n$ polynomials has dimension $n-1$, this theorem completely answers the question, which collections of orbits have independent multipliers in the space of polynomial maps.

\subsection{Outline of the proof}
We prove Theorem~\ref{main1} and Theorem~\ref{main2} in the following way: we consider the space of degree $n$ rational maps with $2n-2$ distinct marked periodic orbits of given periods. This space is a ramified cover over the space $\Rat_n$ of all degree $n$ rational maps. First, in Section~\ref{marked_sp_sec} we prove that this space is an irreducible algebraic set. The multipliers that we consider, are algebraic functions on this set. Then in subsequent sections we show that under conditions of Theorem~\ref{main1} or Theorem~\ref{main2} there exists a point in this set such that the differentials of the multipliers at this point are linearly independent, which implies the desired algebraic independence of the multipliers. In the case of Theorem~\ref{main1}, the rational map that corresponds to this point, is constructed using matings. In the case of Theorem~\ref{main2} we show that the corresponding rational map can be $f_0(z)=z^n$.


\section{The space of polynomials with $k$ marked periodic orbits}\label{marked_sp_sec}

For $n\ge 1$, there is a natural injective map from $\Rat_n$ to $\bbCP^{2n-1}$ defined so that if
$$
f(z)=\frac{p(z)}{q(z)}=\frac{a_nz^n+\dots+a_0}{b_nz^n+\dots+b_0},
$$
then $f$ is mapped to $(a_0:\dots:a_n:b_0:\dots:b_n)\in\bbCP^{2n-1}$. The image of this map is the complement of a certain hypersurface $R_n$ in the projective space $\bbCP^{2n-1}$, where $R_n$ is the zero locus of the resultant of $p$ and $q$. Thus the space $\Rat_n$ can be identified with a Zariski open subset $\bbCP^{2n-1}\setminus R_n$ of the projective $(2n-1)$-space, hence, is an irreducible quasiprojective variety.

For a positive integer $k$, consider a rational map $f\in\Rat_n$ and its $k$ non-multiple periodic points $z_1,\dots,z_k\in\hat\bbC$ belonging to different periodic orbits of (minimal) periods $m_1,\dots,m_k$ respectively. By $\bfm$ denote the vector of periods 
$$
\bfm=(m_1,\dots,m_k).
$$
With any such rational map $f$ and its periodic points belonging to different periodic orbits, one can associate the set $N^n_\bfm$ defined in the following way:

\begin{definition}\label{Nnm_def}
The set $N^n_\bfm=N^n_\bfm(f,z_1,\dots,z_k)$ is the maximal irreducible analytic subset of $\Rat_n\times\hat\bbC^k$, such that

(i) $(f,z_1,\dots,z_k)\in N^n_\bfm$;

(ii) For $(g,w_1,\dots,w_k)\in N^n_\bfm$, the points $w_1,\dots,w_k$ satisfy the equations $g^{\circ m_j}(w_j)=w_j$, for any $j=1,2,\dots,k$.
\end{definition}

Let $\pi\colon N^n_\bfm\to \Rat_n$ be the natural projection
$$
\pi\colon (q, w_1, \dots,w_k)\mapsto q.
$$


\begin{remark}\label{irred_rem}
Since the relations in condition (ii) of Definition~\ref{Nnm_def} are essentially polynomial, it follows that together with the projection $\pi$ the set $N^n_\bfm=N^n_\bfm(f,z_1,\dots, z_k)$ is a ramified cover over $\Rat_n$ and is an irreducible quasiprojective variety. 
\end{remark}

A priori it is not obvious whether the sets $N^n_\bfm$ can be different for different initial choices of $(f, z_1,\dots, z_k)$. We will prove the following lemma, which says that all these sets are the same. 
\begin{lemma}\label{Nnkm_lemma}
Assume that $n\ge 2$. 
Then the set $N^n_\bfm$ is completely determined by the integer $n$, and the vector $\bfm$. The set $N^n_\bfm$ can be described as the closure in $\Rat_n\times\hat\bbC^k$ of the set of all points $(f,z_1,\dots,z_k)\in\Rat_n\times\hat\bbC^k$, where $f\in\Rat_n$ and all $z_j$ are non-multiple periodic points of $f$, belonging to different periodic orbits of corresponding periods $m_j$.
\end{lemma}
\begin{proof}
Consider a rational map $f\in\Rat_n$ that does not have multiple periodic orbits. Given $\bfm=(m_1,\dots,m_k)$, let $z_1,\dots,z_k$ be periodic points of $f$ belonging to different periodic orbits of corresponding periods $m_1,\dots,m_k$. 
It follows from Definition~\ref{Nnm_def} and Remark~\ref{irred_rem} that in order to prove Lemma~\ref{Nnkm_lemma}, it is sufficient to show that every tuple of periodic points $w_1,\dots,w_k$ that belong to different periodic orbits of $f$ with corresponding periods $m_1,\dots, m_k$, can be obtained by analytic continuation of the periodic points $z_1,\dots, z_k$ along some loop $\gamma\subset\Rat_n$. 

We deduce the existence of such a loop $\gamma$ from a similar result (Lemma~1.5 from~\cite{mult_n}) in which rational maps are substituted by polynomials. A minor difficulty is that Lemma~1.5 from~\cite{mult_n} deals with analytic continuation of finite periodic points, while we allow the points $z_j$ and $w_j$ to be infinite. On the other hand, for every non-constant polynomial, infinity is a periodic point of period~$1$, which means that we need to separately consider analytic continuations of fixed points while for periodic points of higher period we can still apply Lemma~1.5 from~\cite{mult_n}. 

As it was mentioned above, first, we prove the existence of the loop $\gamma$ for the case when $\bfm=(1,\dots,1)$. After conjugation by an appropriate M\"{o}bius transformation we may assume that infinity is not a fixed point of $f$. Since $f(z)=p(z)/q(z)$ is a rational map of degree $n$ without multiple periodic orbits, it has precisely $n+1$ fixed points (so $k$ cannot be greater than $n+1$) that are roots of the degree $n+1$ polynomial $zq(z)-p(z)$. Any permutation of these roots can be obtained by analytic continuation along an appropriate loop $\hat\gamma$ in the space of degree $n+1$ polynomials. From such a loop one can reconstruct a loop $\gamma$ in the space $\Rat_n$ that is mapped bijectively onto $\hat{\gamma}$ by the map $p(z)/q(z)\mapsto zq(z)-p(z)$. This way we realize any permutation of the fixed points of $f$.

Now we consider the case of an arbitrary vector $\bfm$. If some of the entries of $\bfm$ are equal to $1$, which corresponds to fixed points, then, as explained in the previous paragraph, we can make a loop in $\Rat_n$ that brings the fixed points from the set $\{z_1,\dots,z_k\}$ to the corresponding fixed points from the set $\{w_1,\dots,w_k\}$. Thus, without loss of generality we may assume that if a point $z_j$ is a fixed point of $f$, then $z_j=w_j$.

We choose a curve $\gamma_1\subset\Rat_n$ that connects $f$ with the polynomial $f_0(z)=z^n$. Analytic continuation of periodic points $z_1,\dots,z_k$ and $w_1,\dots,w_k$ along $\gamma_1$ brings them to corresponding periodic points $z_1',\dots,z_k'$ and $w_1',\dots,w_k'$ of $f_0$. Notice that the points $z_1',\dots,z_k'$ belong to different periodic orbits and similarly the points $w_1',\dots,w_k'$ also belong to different periodic orbits. Moreover, all periodic points of $f_0$ of period greater than~$1$, are finite, so according to Lemma~1.5 from~\cite{mult_n}, there exists a loop $\gamma_2$ in the space of degree $n$ polynomials that begins and ends at the polynomial $f_0$, and analytic continuation of periodic points along $\gamma_2$ brings $z_1',\dots,z_k'$ to the corresponding points $w_1',\dots,w_k'$. Now the loop $\gamma$ can be constructed from $\gamma_1$ and $\gamma_2$ first by going along $\gamma_1$, then along $\gamma_2$, and then returning to $f$ along $-\gamma_1$.
\end{proof}

\subsection{The multiplier map}



Lemma~\ref{Nnkm_lemma} implies that given a vector of periods $\bfm=(m_1,\dots,m_{k})$, the multipliers of all tuples of $k$ distinct periodic orbits with corresponding periods $m_1,\dots,m_k$ are simultaneously either algebraically independent or algebraically dependent over~$\bbC$. This statement together with a sufficient condition for algebraic independence is formulated in the following proposition.
\begin{proposition}\label{suf_cond_prop}
For $n\ge 2$, let $\bfm=(m_1,\dots,m_{k})$ be the vector of periods. If there exists a rational map $g\in\Rat_n$ with $k$ non-multiple periodic points of corresponding periods $m_1,\dots,m_k$, such that the multipliers of these periodic points considered as algebraic functions on $\Rat_n$, are locally independent at $g$, then the multipliers of any $k$ distinct periodic orbits with corresponding periods $m_1,\dots,m_{k}$ considered as (multiple valued) functions on $\Rat_n$, are algebraically independent over~$\bbC$.
\end{proposition}
\begin{proof}
We define the multiplier map $\Lambda\colon N^n_\bfm\to\bbC^{k}$ that with every point $(f,z_1,\dots,z_{k})\in N^n_\bfm$ associates the vector of multipliers of periodic points $z_1,\dots, z_{k}$:
$$
\Lambda\colon (f,z_1,\dots,z_{k})\mapsto ((f^{\circ m_1}(z_1))', (f^{\circ m_2}(z_2))',\dots,(f^{\circ m_{k}}(z_{k}))').
$$

It follows form Lemma~\ref{Nnkm_lemma} that the multipliers of any $k$ distinct periodic orbits of periods $m_1,\dots,m_{k}$ considered as (multiple valued) functions on $\Rat_n$, can be obtained from the multiplier map $\Lambda\colon N^n_\bfm\to\bbC^{k}$ by precomposition with a suitable inverse branch $\pi^{-1}$ of the projection $\pi\colon N^n_\bfm\to\Rat_n$. 
Consider an inverse branch $\pi^{-1}$, such that $\pi^{-1}(g)$ is equal to the map $g$ with a tuple of $k$ non-multiple periodic points with locally independent multipliers. This means that $\Lambda\circ\pi^{-1}$ is locally surjective at $g$, which implies that $\Lambda$ is locally surjective at $\pi^{-1}(g)$.

According to Definition~\ref{Nnm_def}, the set $N^n_\bfm$ is irreducible, and since the multiplier map $\Lambda$ is an algebraic map on $N^n_\bfm$ that is locally surjective at one point, it follows that $\Lambda$ is locally surjective everywhere outside of some codimension~$1$ subset of $N^n_\bfm$.

Since every branch of $\pi^{-1}$ is a local diffeomorphism everywhere outside of some codimension~$1$ subset of $\Rat_n$, the composition $\Lambda\circ\pi^{-1}$ is always locally surjective at least at one point of $\Rat_n$, which implies that the multipliers of any $k$ distinct periodic orbits of corresponding periods $m_1,\dots,m_k$ are algebraically independent (multiple valued) functions on $\Rat_n$.
\end{proof}



\section{The case of quadratic rational maps}
In this section we will use Proposition~\ref{suf_cond_prop} to prove Theorem~\ref{main1}, which deals with the case $n=2$. 
\begin{proof}[Proof of Theorem~\ref{main1}]
A direct computation shows that every quadratic rational map has no more than one periodic orbit of period~$2$. Thus, in a pair of distinct periodic orbits at least one has period different from~$2$.

Now according to Proposition~\ref{suf_cond_prop}, in order to prove Theorem~\ref{main1}, it is sufficient for every pair of periods $m_1, m_2\ge 1$ that are not simultaneously equal to~$2$, to find a rational map $f\in\Rat_2$ that has two periodic orbits with corresponding periods $m_1, m_2$ and locally independent multipliers. We notice that due to the ideas of quasiconformal surgery, attracting periodic orbits of a rational map will always have locally independent multipliers (e.g. see~\cite{Milnor_hyper} for a proof). Thus, in order to prove Theorem~\ref{main1}, it is sufficient to show that for any pair of periods $m_1, m_2$ that are not simultaneously equal to~$2$, there exists a quadratic rational map with two distinct attracting periodic orbits of periods $m_1$ and $m_2$.

We will construct such a rational map by mating two quadratic polynomials with attracting periodic orbits of periods $m_1$ and $m_2$. Since $m_1$ and $m_2$ are not simultaneously equal to~$2$, one can always choose two hyperbolic polynomials $p_1(z)=z^2+c_1$ and $p_2(z)=z^2+c_2$, such that polynomial $p_1$ has an attracting periodic orbit of period $m_1$, polynomial $p_2$ has an attracting periodic orbit of period $m_2$, and the parameters $c_1$ and $c_2$ do not lie in conjugate limbs of the Mandelbrot set. Then according to~\cite{TanLei92} and~\cite{Shishikura_2000}, there exists a quadratic rational map that is a mating of $p_1$ and $p_2$. In particular, this rational map will have two distinct attracting periodic orbits of periods $m_1$ and $m_2$, which finishes the proof.
\end{proof}

\section{Computation of derivatives}\label{deriv_sec}
The rest of the paper is devoted to showing that for $n\ge 3$ and for any combination of $2n-2$ periods satisfying Theorem~\ref{main2}, the map $f_0(z)=z^n$ has $2n-2$ locally independent periodic orbits of those periods. Then because of Proposition~\ref{suf_cond_prop}, this will imply Theorem~\ref{main2}.

Given a positive integer $n\ge 2$, we consider the family of degree $n$ rational maps
\begin{equation}\label{family}
f_\bfa(z)=\frac{z^n+a_{n-2}z^{n-2}+\dots+a_1z+a_0}{1-a_{n+1}z-a_{n+2}z^2-\dots-a_{2n-1}z^{n-1}},
\end{equation}
parameterized by the $(2n-2)$-dimensional parameter $\bfa=(a_0,\dots,a_{n-2},a_{n+1},\dots,a_{2n-1})\in\bbC^{2n-2}$. 
For $\bfa=0$ the corresponding map $f_0(z)=z^n$ does not have multiple periodic orbits so by the Implicit Function Theorem, every periodic point $z_0\in\hat\bbC$ of period $m$ for the map $f_0$ defines a unique locally analytic function $z(\bfa)$ -- the periodic point of a nearby map $f_\bfa$, such that $z(0)=z_0$ and $f_\bfa^{\circ m}(z(\bfa))=z(\bfa)$, for $\bfa$ from some neighborhood of the origin. Then in that neighborhood of the origin one can consider the multiplier map 
$$
\lambda_{z_0}(\bfa)=(f_\bfa^{\circ m})'(z(\bfa)).
$$
In the following lemma we compute the derivatives of $\lambda_{z_0}(\bfa)$ with respect to different coordinates $a_j$ of the vector $\bfa$.

\begin{lemma}\label{comp_deriv_lemma}
For $n\ge 2$, let $z_0\neq 0,\infty$ be a periodic point of period $m$ of the map $f_0(z)=z^n$. Then for any index $j$ satisfying $0\le j\le 2n-1$ and $j\neq n-1$, $j\neq n$, the following holds:
\begin{equation*}
\frac{d\lambda_{z_0}(0)}{da_j}= (jn^{m-1}-n^m)\sum_{i=0}^{m-1}z_0^{n^i(j-n)}.
\end{equation*}
\end{lemma}
\begin{proof}
First of all, we notice that when $0\le j\le n-2$, the result of Lemma~\ref{comp_deriv_lemma} follows from~\cite[Lemma~3.1]{mult_n}.
In order to formulate, Lemma~3.1 from~\cite{mult_n}, we need the following construction: for a fixed non-negative integer $j$ consider the family of maps 
\begin{equation}\label{f_aj}
f_{a,j}(z)=z^n+az^j,
\end{equation}
parameterized by a single parameter $a\in\bbC$. If $z_0$ is a periodic point of period $m$ of the map $f_0(z)=z^n$, then in the same way as before, we can obtain a locally analytic function $z(a)$ -- the periodic point of a nearby map $f_{a,j}$, such that $z(0)=z_0$ and $f_{a,j}^{\circ m}(z(a))=z(a)$. 
\begin{lemma}\label{cit_lemma} \cite[Lemma~3.1]{mult_n}
For $n\ge 2$, let $z_0\neq 0,\infty$ be a periodic point of period $m$ of the map $f_0(z)=z^n$. Then for any non-negative integer $j$ and the corresponding multiplier map $\mu_{j,z_0}(a)=(f_{a,j}^{\circ m})'(z(a))$ the following holds:
\begin{equation*}
\frac{d\mu_{j,z_0}(0)}{da}= (jn^{m-1}-n^m)\sum_{i=0}^{m-1}z_0^{n^i(j-n)}.
\end{equation*}
\end{lemma}

\begin{remark}
We note that the original formulation of Lemma~\ref{cit_lemma} contained an additional restriction for the possible values of $j$. That restriction was not used in the proof of the Lemma and appeared only in order to comply with a specific construction considered in that paper.
\end{remark}

Now we return to the proof of the second case of Lemma~\ref{comp_deriv_lemma}, namely, when $j$ satisfies $n+1\le j\le 2n-1$. For all $k\neq j$ we fix $a_k=0$, while we allow $a_j$ to change. In other words, we consider the vector of parameters $\bfa$ in the form 
\begin{equation}\label{bfa_0_form}
\bfa=(0,0,\dots,a_j,0,\dots,0).
\end{equation}
Then for all sufficiently small $a_j$ and for all $z$ such that $|z|<2$, the function $f_\bfa(z)$ can be expressed as a convergent series
$$
f_\bfa(z)= \frac{z^n}{1-a_jz^{j-n}}=z^n+a_jz^{j}+o(a_j)=f_{a_j,j}(z)+o(a_j),
$$
where $f_{a_j,j}(z)$ is defined in~(\ref{f_aj}) and by $o(a_j)$ we denote the terms that contain $a_j$ in the power greater than~$1$. In particular, since all bounded nonzero periodic points of the map $f_0(z)=z^n$ have modulus~$1$, then for all vectors $\bfa$ of the form~(\ref{bfa_0_form}) with sufficiently small $a_j$ we have
$$
f_\bfa(z(\bfa))= f_{a_j,j}(z(\bfa))+o(a_j).
$$

Finally, from the previous identity it is not hard to see that 
$$
\left.\frac{d}{da_j} (f_\bfa^{\circ m})'(z(\bfa))\right|_{\bfa=0}= \left.\frac{d}{da_j} (f_{a_j,j}^{\circ m})'(z(\bfa))\right|_{\bfa=0},
$$
so
$$
\frac{d\lambda_{z_0}(0)}{da_j}=\frac{d\mu_{j,z_0}(0)}{da_j},
$$
and Lemma~\ref{comp_deriv_lemma} follows from Lemma~\ref{cit_lemma}.



\end{proof}

\begin{proposition}\label{poly_prop}
For every positive integers $n$, $m$ with $n\ge 2$ and every index $j$ satisfying $0\le j\le 2n-1$ and $j\neq n-1$, $j\neq n$, there exists a nonzero polynomial $P_{n,j,m}(z)$, such that if $z_0\neq 0,\infty$ is a periodic point of period $m$ of the function $f_0(z)=z^n$, then
$$
\frac{d\lambda_{z_0}(0)}{da_j}= \frac{1}{z_0^{n^{m-1}}}P_{n,j,m}\left(z_0\right).
$$
Moreover,
\begin{equation}\label{deg_eq}
\deg P_{n,j,m}=
\begin{cases}
(j+1)n^{m-1}-1, &\text{for $0\le j\le n-2$,} \\
(j-n+1)n^{m-1}, &\text{for $n+1\le j\le 2n-2$,} \\
2n^{m-1}-n^{m-2}, &\text{for $j=2n-1$ and $m\ge 2$,} \\
1, &\text{for $j=2n-1$ and $m=1$},
\end{cases}
\end{equation}
and the following properties hold:

(a) No two polynomials $P_{n,j_1,m}$ and $P_{n,j_2,m}$, where $0\le j_1<j_2\le n-2$, have terms of the same degree.

(b) If $m\ge 2$, then no two polynomials $P_{n,j_1,m}$ and $P_{n,j_2,m}$, where $0\le j_1<j_2\le 2n-2$ and $j_1,j_2\neq n, n-1$, have terms of the same degree.

(c) If $m\ge 3$, then no two polynomials $P_{n,j_1,m}$ and $P_{n,j_2,m}$, where $0\le j_1<j_2\le 2n-1$ and $j_1,j_2\neq n, n-1$, have terms of the same degree.
\end{proposition}
\begin{proof}
The proof of Proposition~\ref{poly_prop} follows from Lemma~\ref{comp_deriv_lemma} and the fact that $z_0^{n^m}=z_0$. First we compute the polynomials $P_{n,j,m}$ and their degrees.

For $j=0,\dots, n-2$,
\begin{multline*}
\frac{d\lambda_{z_0}(0)}{da_j}= \frac{jn^{m-1}-n^m}{z_0^{n^{m-1}}}\sum_{i=0}^{m-1}z_0^{n^i(j-n)+n^{m-1}}= \\
\frac{jn^{m-1}-n^m}{z_0^{n^{m-1}}}\left(z_0^{(j+1)n^{m-1}-1}+\sum_{i=0}^{m-2}z_0^{n^i(j-n)+n^{m-1}}\right),
\end{multline*}
so
\begin{equation}\label{Pj1}
P_{n,j,m}(z)=(jn^{m-1}-n^m)\left(z^{(j+1)n^{m-1}-1}+\sum_{i=0}^{m-2}z^{n^i(j-n)+n^{m-1}}\right),
\end{equation}
and $\deg P_{n,j,m}= (j+1)n^{m-1}-1$.


For $j=n+1,\dots, 2n-2$, 
$$
\frac{d\lambda_{z_0}(0)}{da_j}= \frac{jn^{m-1}-n^m}{z_0^{n^{m-1}}}\sum_{i=0}^{m-1}z_0^{n^i(j-n)+n^{m-1}},
$$
so
\begin{equation}\label{Pj2}
P_{n,j,m}(z)=(jn^{m-1}-n^m)\sum_{i=0}^{m-1}z^{n^i(j-n)+n^{m-1}},
\end{equation}
and $\deg P_{n,j,m}=(j-n+1)n^{m-1}$.

For $j=2n-1$,
$$
\frac{d\lambda_{z_0}(0)}{da_j}= \frac{jn^{m-1}-n^m}{z_0^{n^{m-1}}}\sum_{i=0}^{m-1}z_0^{n^i(n-1)+n^{m-1}}=
\frac{jn^{m-1}-n^m}{z_0^{n^{m-1}}}\left(z_0+\sum_{i=0}^{m-2}z_0^{n^i(n-1)+n^{m-1}}\right),
$$
so
\begin{equation}\label{Pj3}
P_{n,j,m}(z)=(jn^{m-1}-n^m)\left(z+\sum_{i=0}^{m-2}z^{n^i(n-1)+n^{m-1}}\right).
\end{equation}
It follows from~(\ref{Pj3}) that if $m=1$, then $\deg P_{n,j,m}=1$, and if $m\ge 2$, then $\deg P_{n,j,m}=2n^{m-1}-n^{m-2}$.

Now we will prove properties (a),~(b), and~(c). According to~(\ref{Pj1}), when $m=1$, the polynomial $P_{n,j,m}(z)$ consists of one monomial of degree ${(j+1)n^{m-1}-1}$, which immediately implies property~(a) for $m=1$. Thus, since properties (b) and (c) require $m$ to be greater than~$1$, we may further assume without loss of generality that $m\ge 2$.

First of all, we notice that the terms of the form $z^{n^i(j-n)+n^{m-1}}$ from~(\ref{Pj1}),~(\ref{Pj2}) and~(\ref{Pj3}) have different degrees for different pairs $(i, j)$ with $0\le j\le 2n-1$ and $j \neq n, n-1$. Indeed, assume that for two pairs $(i_1,j_1)$ and $(i_2,j_2)$, the degrees match. This means that 
\begin{equation}\label{contradict1}
n^{i_1-i_2}(j_1-n)=j_2-n.
\end{equation}
If $i_1=i_2$, this implies that $j_1=j_2$. If $i_1\neq i_2$, then we can assume that $i_1>i_2$, and~(\ref{contradict1}) implies that $n$ divides $j_2$. Thus, $j_2=0$ and~(\ref{contradict1}) is transformed into the identity $n^{i_1-i_2-1}(j_1-n)=-1$. The latter is not possible for $i_1>i_2$ and $j_1\neq n-1$.

Next, we compare the degrees of the terms of the form $z^{n^i(j-n)+n^{m-1}}$ with the degrees of the terms $z^{(j+1)n^{m-1}-1}$ from~(\ref{Pj1}). Assume that $(j_1+1)n^{m-1}-1= n^i(j_2-n)+n^{m-1}$ for some nonnegative integers $j_1, j_2, i$ with $0\le j_1,j_2\le 2n-1$ and $j_1,j_2\neq n, n-1$. According to our assumption, $m\ge 2$, so $n$ does not divide the left hand side of the above identity, so it follows that $i=0$ and 
\begin{equation}\label{contradict2}
j_1n^{m-1}=j_2+1-n
\end{equation}
Identity~(\ref{contradict2}) implies that $n$ divides $j_2+1$, and since $0\le j_2\le 2n-1$ and $j_2\neq n-1$, we have the only possibility $j_2=2n-1$. This finishes the proof of properties~(a) and~(b).

Finally we notice that if $j_2=2n-1$, then identity~(\ref{contradict2}) does not hold for $m\ge 3$. As the last step, it is not hard to see that for $m\ge 3$ no terms in~(\ref{Pj1}),~(\ref{Pj2}) and~(\ref{Pj3}) except for the first term in~(\ref{Pj3}) have degree~$1$. This finishes the proof of property~(c).
\end{proof}

Now we separately consider the multiplier at infinity. Notice that infinity is a non-multiple fixed point of the map $f_0(z)=z^n$. Moreover, infinity is a fixed point of every map from the family~(\ref{family}), so we can consider the corresponding multiplier map $\lambda_\infty(\bfa)=f_\bfa'(\infty)$.
\begin{lemma}\label{diff_inf_lemma}
For $n\ge 2$ and for any index $j$ satisfying $0\le j\le 2n-1$ and $j\neq n-1$, $j\neq n$, the following holds:
\begin{equation*} 
\frac{d\lambda_{\infty}(0)}{da_j}=
\begin{cases}
-1, &\text{for $j=2n-1$,} \\
0, &\text{for $j\neq 2n-1$}.
\end{cases}
\end{equation*}
\end{lemma}
\begin{proof}
After the coordinate change $z\mapsto 1/z$, the map $f_\bfa(z)$ takes the form
$$
g_\bfa(z)=\frac{z^n-a_{n+1}z^{n-1}-\dots-a_{2n-1}z}{1+a_{n-2}z^2+\dots+a_1z^{n-1}+a_0z^n}.
$$
Then it follows that $\lambda_\infty(\bfa)=g_\bfa'(0)=-a_{2n-1}$, and the result of the lemma becomes evident.
\end{proof}

\section{The number of periodic points}
For $n\ge 2$, let $\nu_n(m)$ denote the number of bounded periodic points of the polynomial map $f_0(z)=z^n$ with period $m$. Since this polynomial does not have multiple periodic points, the function $\nu_n(m)$ can be computed inductively by the formula
$$
n^m=\sum_{r|m}\nu_n(r),\qquad\text{or}\quad \nu_n(m)=\sum_{r|m}\mu(m/r)n^r,
$$
where the summation goes over all divisors $r\ge 1$ of $m$, and $\mu(m/r)\in\{\pm 1,0\}$ is the M\"{o}bius function defined by
$$
\mu(p_1\dots p_k)=(-1)^k
$$
for a product of $k\ge 0$ distinct primes and $\mu(m)=0$, if $m$ is not a product of distinct primes.

It is easy to see from these formulas that 
\begin{equation*} 
\nu_n(m)\ge n^m-n^{m-2}, \quad\text{for }m\ge 3,\qquad\text{and}
\end{equation*}
$$
\nu_n(1)=n,\quad \nu_n(2)=n^2-n.
$$

Let $\hat\nu_n(m)$ denote the number of bounded non-zero periodic points of the polynomial $f_0(z)=z^n$ with period $m$. Then, since zero is a fixed point of the polynomial $f_0$, it follows that $\hat\nu_n(m)=\nu_n(m)$, for $m>1$ and $\hat\nu_n(1)=\nu(1)-1$. Thus from the previous relations on $\nu_n(m)$ we obtain
\begin{equation}
\begin{array}{c}\label{nu_k3}
\hat\nu_n(m)\ge n^m-n^{m-2}, \quad\text{for }m\ge 3,\qquad\text{and} \\
\hat\nu_n(1)=n-1,\quad \hat\nu_n(2)=n^2-n.
\end{array}
\end{equation}

\section{Inductive arguments}\label{induct_sec}

\begin{definition}
Given a $(2n-2)$-dimensional vector of points
\begin{equation}\label{z_vect}
\bfz=(z_0,z_1,\dots,z_{n-2},z_{n+1},\dots,z_{2n-1})\in\hat\bbC^{2n-2},
\end{equation}
and a $(2n-2)$-dimensional vector of of periods
\begin{equation}\label{m_vect}
\bfm=(m_0,\dots,m_{n-2},\,m_{n+1},\dots,m_{2n-1})\in\bbN^{2n-2},
\end{equation}
we will say that $\bfz$ is an $\bfm$\emph{-periodic vector of a map} $f$, if each element $z_j$ is a periodic point of $f$ of corresponding period $m_j$.
We will say that $\bfz$ is simply an $\bfm$-\emph{periodic vector}, if it is an $\bfm$-periodic vector of the map $f_0(z)=z^n$.
\end{definition}


According to Section~\ref{deriv_sec}, if $\bfz$ is an $\bfm$-periodic vector, then for each $z_j$ one can consider the corresponding multiplier map $\lambda_{z_j}(\bfa)$ defined for all sufficiently small vectors $\bfa$. Thus 
we can define a function $\Lambda_\bfz\colon (\bbC^{2n-2}, 0)\to\bbC^{2n-2}$, such that
\begin{equation*}
\Lambda_\bfz(\bfa)=(\lambda_{z_0}(\bfa),\dots,\lambda_{z_{n-2}}(\bfa),\, \lambda_{z_{n+1}}(\bfa),\dots,\lambda_{z_{2n-2}}(\bfa)).
\end{equation*}

\begin{definition}\label{minor_def}
(a) For each $j$, satisfying $0\le j\le 2n-1$ and $j\neq n-1$, $j\neq n$, by $\mcJ_\bfz(j)$ we denote the submatrix of the Jacobi matrix $\frac{d\Lambda_\bfz}{d\bfa}(0)$, obtained from $\frac{d\Lambda_\bfz}{d\bfa}(0)$ by deleting all columns and rows that are located to the right and below the diagonal element $\frac{\partial\lambda_{z_j}}{\partial a_j}(0)$.

(b) For convenience of notation we define $\mcJ_\bfz(n)=\mcJ_\bfz(n-2)$.
\end{definition}



The goal of this Section is to construct inductive arguments which under some restrictions on the vector $\bfm$ will allow us to prove existence of an $\bfm$-periodic vector $\bfz$, such that all matrices $\mcJ_\bfz(j)$ are non-degenerate.

The following proposition will serve as the base of our induction:
\begin{proposition}\label{base_prop}
For $n\ge 2$ and for any vector of periods $\bfm\in\bbN^{2n-2}$ there exists an $\bfm$-periodic vector $\bfz$, such that $\mcJ_\bfz(0)$ is non-degenerate.
\end{proposition}
\begin{proof}
We notice that $\mcJ_\bfz(0)$ is a $(1\times 1)$-matrix $\frac{d\lambda_{z_0}}{da_0}(0)$, so in order to prove the Proposition, it is sufficient to show that there exists a periodic point $z_0$ of period $m_0$ for the map $f_0(z)=z^n$, such that $\frac{d\lambda_{z_0}}{da_0}(0)\neq 0$.

According to Proposition~\ref{poly_prop}, if $z_0\neq 0,\infty$, then 
$$
\frac{d\lambda_{z_0}(0)}{da_0}= \frac{1}{z_0^{n^{m_0-1}}}P_{n,0,m_0}\left(z_0\right),
$$
where $P_{n,0,m_0}$ is a non-identically zero polynomial of degree $\deg P_{n,0,m_0}=n^{m_0-1}-1$. Notice that according to~(\ref{nu_k3}), the number of bounded nonzero periodic points of period $m_0$ for the map $f_0(z)=z^n$ is equal to $\hat{\nu}_n(m_0)>\deg P_{n,0,m_0}=n^{m_0-1}-1$. Thus, there always exists a periodic point $z_0$, for which $P_{n,0,m_0}\left(z_0\right)\neq 0$. This finishes the proof.
\end{proof}




The next two lemmas will constitute the step of our induction.
\begin{lemma}\label{step1_lemma}
Assume that $n\ge 2$ and the vector of periods $\bfm$ in~(\ref{m_vect}) is such that $m_{n+1},\dots,m_{2n-3}\ge 2$ and $m_{2n-2}\ge 3$. Assume that there exists an $\bfm$-periodic vector $\bfz$, such that the matrix $\mcJ_\bfz(j-1)$ is non-degenerate for some $j$ that satisfies $1\le j\le 2n-2$, $j\neq n-1$, $j\neq n$. Then the vector $\bfz$ can be chosen in such a way that the matrix $\mcJ_\bfz(j)$ is also non-degenerate.
\end{lemma}

\begin{lemma}\label{step2_lemma}
Assume that $n\ge 3$ and the vector of periods $\bfm$ in~(\ref{m_vect}) is such that $m_{2n-1}\neq 2$. Assume that there exists an $\bfm$-periodic vector $\bfz$, such that the matrix $\mcJ_\bfz(2n-2)$ is non-degenerate. Then the vector $\bfz$ can be chosen in such a way that the matrix $\mcJ_\bfz(2n-1)=\frac{d\Lambda_\bfz}{d\bfa}(0)$ is also non-degenerate.
\end{lemma}

\begin{proof}[Proof of Lemma~\ref{step1_lemma}]
We notice that the coordinates of the vector $\bfz$ with indexes greater than $j$ do not appear in the matrix $\mcJ_\bfz(j)$, so these coordinates have no effect on the degeneracy or non-degeneracy of the matrix $\mcJ_\bfz(j)$ and hence, can be chosen arbitrarily. Assuming that $\bfz$ is chosen in such a way that $\mcJ_\bfz(j-1)$ is non-degenerate, we will adjust the coordinate $z_j$ so that $\mcJ_\bfz(j)$ would also be non-degenerate.

According to Definition~\ref{minor_def} (definition of $\mcJ_\bfz(j)$) and Proposition~\ref{poly_prop}, if $z_j\neq 0,\infty$, then the $j$-th row of the matrix $\mcJ_\bfz(j)$ has the form
\begin{equation}\label{last_row}
\left( \frac{1}{z_j^{n^{m_j-1}}}P_{n,0,m_j}(z_j), \dots, \frac{1}{z_j^{n^{m_j-1}}}P_{n,j,m_j}(z_j) \right),
\end{equation}
and $z_j$ does not appear in other entries of matrix $\mcJ_\bfz(j)$. 
Assuming that all coordinates of $\bfz$ except $z_j$ are fixed, we can express the determinant $\det\mcJ_\bfz(j)$ as a function of $z_j$ using the cofactor expansion along the $j$-th row. We obtain that 
\begin{equation}\label{det_P}
\det\mcJ_\bfz(j)=\frac{1}{z_j^{n^{m_j-1}}} P(z_j),
\end{equation}
where $P(z)$ is some linear combination of polynomials $P_{n,0,m_j},\dots,P_{n,j,m_j}$. 

Now we consider three different cases:

\underline{Case 1}: If $1\le j\le n-2$ then according to property (a) of Proposition~\ref{poly_prop}, polynomial $P$ is zero if and only if all cofactors of the matrix $\mcJ_\bfz(j)$ along the $j$-th row are equal to zero. However, this does not happen since the matrix $\mcJ_\bfz(j-1)$ is non-degenerate. Hence, polynomial $P$ is not identically zero.

Now using~(\ref{deg_eq}) and~(\ref{nu_k3}) we obtain the following estimates on the degree of $P$:
$$
\deg P\le \max_{0\le i\le j} \deg P_{n,i,m_j}\le n^{m_j}-n^{m_j-1}-1<\hat{\nu}_n(m_j),
$$
where $\hat\nu_n(m_j)$ is the number of bounded non-zero periodic points of the polynomial $f_0(z)=z^n$ with period $m_j$. Since polynomial $P$ is not identically zero, the above estimate implies that $z_j$ can be chosen in such a way that $P(z_j)\neq 0$ and  the matrix $\mcJ_\bfz(j)$ is non-degenerate.

\underline{Case 2}: If $n+1\le j\le 2n-3$, then according to the conditions of Lemma~\ref{step1_lemma}, we have $m_j\ge 2$. Then in the similar way as in \underline{Case 1}, property (b) of Proposition~\ref{poly_prop} implies that polynomial $P$ is not identically zero.
Similarly to \underline{Case 1}, we have
$$
\deg P\le \max_{0\le i\le j,i\neq n-1, n} \deg P_{n,i,m_j} = n^{m_j}-n^{m_j-1}-1<\hat{\nu}_n(m_j),
$$
so $z_j$ can be chosen in such a way that $P(z_j)\neq 0$ and the matrix $\mcJ_\bfz(j)$ is non-degenerate.

\underline{Case 3}: If $j=2n-2$, then polynomial $P$ is not identically zero exactly for the same reasons as in \underline{Case 2}.
Using~(\ref{deg_eq}),~(\ref{nu_k3}) and the fact that $m_{2n-2}\ge 3$, we obtain that
$$
\deg P\le \max_{0\le i\le 2n-2,i\neq n-1, n} \deg P_{n,i,m_{2n-2}} = n^{m_{2n-2}}-n^{m_{2n-2}-1}<\hat{\nu}_n(m_{2n-2}),
$$
so $z_{2n-2}$ can be chosen in such a way that $P(z_{2n-2})\neq 0$ and the matrix $\mcJ_\bfz(2n-2)$ is non-degenerate.
\end{proof}

\begin{proof}[Proof of Lemma~\ref{step2_lemma}]
We will show that we can adjust the last coordinate $z_{2n-1}$ of the vector $\bfz$ so that the Jacobi matrix $\frac{d\Lambda_\bfz}{d\bfa}(0)=\mcJ_\bfz(2n-1)$ is non-degenerate.

First, if $m_{2n-1}=1$, then we can choose $z_{2n-1}=\infty$. Then, according to Lemma~\ref{diff_inf_lemma}, the last row of the matrix $\mcJ_\bfz(2n-1)$ consists of all zeros except the rightmost element that is equal to $-1$. Thus $\det \mcJ_\bfz(2n-1)= -\det\mcJ_\bfz(2n-2)$, and since matrix $\mcJ_\bfz(2n-2)$ is non-degenerate, the matrix $\mcJ_\bfz(2n-1)$ is non-degenerate as well.

Now if $m_{2n-1}\ge 3$, then similarly to the proof of Lemma~\ref{step1_lemma}, the last row of the matrix $\mcJ_\bfz(2n-1)$ has the form~(\ref{last_row}) for $j=2n-1$
and $\det\mcJ_\bfz(2n-1)$ can be expressed in the form~(\ref{det_P}), where $j=2n-1$ and $P$ is a linear combination of polynomials $P_{n,0,m_{2n-1}},\dots, P_{n,2n-1,m_{2n-1}}$. From property (c) of Proposition~\ref{poly_prop} and the fact that $\mcJ_\bfz(2n-2)$ is a non-degenerate matrix, we conclude that polynomial $P$ is not identically zero. Using~(\ref{deg_eq}),~(\ref{nu_k3}) and the conditions that $n\ge 3$ and $m_{2n-1}\ge 3$, we obtain that
\begin{multline*}
\deg P\le \max_{0\le i\le 2n-1,i\neq n-1, n} \deg P_{n,i,m_{2n-1}} = \\
\max(n^{m_{2n-1}}-n^{m_{2n-1}-1},\, 2n^{m_{2n-1}-1}-n^{m_{2n-1}-2})= 
n^{m_{2n-1}}-n^{m_{2n-1}-1} <\hat{\nu}_n(m_{2n-1}),
\end{multline*}
so $z_{2n-1}$ can be chosen in such a way that $P(z_{2n-1})\neq 0$ and the matrix $\mcJ_\bfz(2n-1)$ is non-degenerate.
\end{proof}

\section{Existence of independent multipliers}
In this section we put together the inductive arguments from Section~\ref{induct_sec} in order to prove the following Proposition:

\begin{proposition}\label{exist_th}
Assume that $n\ge 3$ and let $\bfm$ be a vector of periods defined in~(\ref{m_vect}) and satisfying the following conditions:

(i) No more than $n$ different coordinates of $\bfm$ are equal to~$1$.

(ii) At least one of the coordinates of $\bfm$ is greater than~$2$ and all other coordinates are not equal to~$2$.

Then there exists an $\bfm$-periodic vector $\bfz$ of the form~(\ref{z_vect}), such that the Jacobi matrix $\frac{d\Lambda_\bfz}{d\bfa}(0)$ is non-degenerate.
\end{proposition}
\begin{proof}
Because of condition~(ii) without loss of generality we may assume that $m_{2n-2}\ge 3$ and $m_{2n-1}\neq 2$. Moreover, if exactly $n$ different coordinates of $\bfm$ are equal to~$1$, then we may assume that $m_{2n-1}=1$. Further we may assume that all other coordinates of $\bfm$ are put in a non-decreasing order. Then condition~(i) implies that $m_j\ge 2$, for all $j$ such that $n+1\le j\le 2n-3$.

Now we finish the proof by induction, where the base of the induction is established by Proposition~\ref{base_prop} and the inductive steps are obtained by applying Lemma~\ref{step1_lemma} and Lemma~\ref{step2_lemma}.
\end{proof}

\begin{proof}[Proof of Theorem~\ref{main2}]
If $2n-2$ periodic orbits satisfy the conditions of Theorem~\ref{main2}, then their periods satisfy the conditions of Proposition~\ref{exist_th}, hence according to Proposition~\ref{exist_th}, there exists a collection of $2n-2$ periodic orbits of the map $f_0$ with the same combination of periods, such that their multipliers are locally independent. Then Theorem~\ref{main2} follows from this fact and Proposition~\ref{suf_cond_prop}.
\end{proof}



\bibliography{rational}

\def\cprime{$'$}
\begin{thebibliography}{1}

\bibitem{mult_n}
I.~Gorbovickis.
\newblock Algebraic independence of multipliers of periodic orbits in the space
  of polynomial maps of one variable.
\newblock arXiv:1305.0867.

\bibitem{McMullen}
C.~McMullen.
\newblock Families of rational maps and iterative root-finding algorithms.
\newblock {\em Ann. of Math. (2)}, 125(3):467--493, 1987.

\bibitem{Milnor_M2}
J.~Milnor.
\newblock Geometry and dynamics of quadratic rational maps.
\newblock {\em Experiment. Math.}, 2(1):37--83, 1993.
\newblock With an appendix by the author and Lei Tan.

\bibitem{Milnor}
J.~Milnor.
\newblock {\em Dynamics in one complex variable}, volume 160 of {\em Annals of
  Mathematics Studies}.
\newblock Princeton University Press, Princeton, NJ, third edition, 2006.

\bibitem{Milnor_hyper}
J.~Milnor.
\newblock Hyperbolic components.
\newblock In {\em Conformal dynamics and hyperbolic geometry}, volume 573 of
  {\em Contemp. Math.}, pages 183--232. Amer. Math. Soc., Providence, RI, 2012.
\newblock With an appendix by A. Poirier.

\bibitem{Shishikura_2000}
M.~Shishikura.
\newblock On a theorem of {M}. {R}ees for matings of polynomials.
\newblock In {\em The {M}andelbrot set, theme and variations}, volume 274 of
  {\em London Math. Soc. Lecture Note Ser.}, pages 289--305. Cambridge Univ.
  Press, Cambridge, 2000.

\bibitem{TanLei92}
L.~Tan.
\newblock Matings of quadratic polynomials.
\newblock {\em Ergodic Theory Dynam. Systems}, 12(3):589--620, 1992.

\end{thebibliography}
\bibliographystyle{abbrv} 

\textsc{Department of Mathematics, University of Toronto, Room 6290, 40~St. George Street, Toronto, Ontario, Canada M5S~2E4}

\emph{E-mail address}: igors.gorbovickis@utoronto.ca

\end{document}